\documentclass[a4paper, reqno]{amsart}
\usepackage{amssymb, amsmath, amsthm, enumerate, mathrsfs, mathscinet, mathtools, standalone, tikz}

\numberwithin{equation}{section}
\newtheorem{theorem}{Theorem}[section]

\newtheorem{lemma}[theorem]{Lemma}
\newtheorem{proposition}[theorem]{Proposition}

\theoremstyle{definition}

\newtheorem{remark}[theorem]{Remark}
\newtheorem{example}[theorem]{Example}

\newtheorem*{acknowledgement}{Acknowledgement}

\newcommand{\C}{\mathbb{C}}
\newcommand{\R}{\mathbb{R}}
\newcommand{\Z}{\mathbb{Z}}
\newcommand{\N}{\mathbb{N}}
\newcommand{\Q}{\mathbb{Q}}
\newcommand{\ud}{\mathrm{d}}

\newcommand{\RN}[1]{%
  \textup{\uppercase\expandafter{\romannumeral#1}}%
}

\begin{document}

\title{An abstract $L^2$ Fourier restriction theorem}

\begin{abstract} An $L^2$ Fourier restriction argument of Bak and Seeger is abstracted to the setting of locally compact abelian groups. This is used to prove new restriction estimates for varieties lying in modules over local fields or rings of integers $\Z/N\Z$. 
\end{abstract}

\author[J. Hickman]{ Jonathan Hickman }
\address{Jonathan Hickman: Eckhart Hall Room 414, Department of Mathematics, University of Chicago, 5734 S. University Avenue, Chicago, Illinois,  60637, US.}
\email{jehickman@uchicago.edu}

\author[J. Wright]{ James Wright }
\address{James Wright: Room 4621, James Clerk Maxwell Building, The King's Buildings, Peter Guthrie Tait Road, Edinburgh, EH9 3FD, UK.}
\email{j.r.wright@ed.ac.uk}

\date{\today}





\maketitle

\section{Introduction}\label{introduction section}

Let $n \geq 2$ and $\Sigma \subseteq \R^n$ be a hypersurface and $\mu$ a smooth, compactly supported density on $\Sigma$. Suppose that the Gaussian curvature of $\Sigma$ does not vanish on the support of $\mu$. The classical Stein-Tomas Fourier restriction theorem \cite{Tomas1975, Stein1993} then asserts that the \emph{a priori} estimate\footnote{If $L$ is a list of objects and $X, Y \geq 0$, then the notation $X \lesssim_L Y$ or $Y \gtrsim_L X$ signifies $X \leq C_L Y$ where $C_L$ is a constant which depends only on the objects featured in $L$.} 
\begin{equation*}
\|\hat{f}|_{\Sigma}\|_{L^2(\Sigma, \mu)} \lesssim_{\mu} \|f\|_{L^r(\R^n)} 
\end{equation*}
is valid for all $1 \leq r \leq 2(n+1)/(n+3)$. A large number of variants and generalisations of this important inequality have appeared in the literature. For instance, one may relax the curvature condition on the hypersurface and prove estimates for a restricted range of $r$, or investigate measures supported on surfaces of larger co-dimension \cite{Christ1985, Drury1985, Greenleaf1981, Magyar2009, Ikromov2011, Ikromov2016}. The underlying surface can be removed entirely by working with abstract measures satisfying certain dimensional and Fourier-dimensional hypotheses \cite{Bak2011, Mitsis2002, Mockenhaupt2000}. Restriction theory can also be formulated in alternative algebraic settings, and in particular for varieties lying in vector spaces over finite fields \cite{Iosevich2008, Iosevich2010, Iosevich2010a, Koh2014, Lewko2012, Lewko2015, Mockenhaupt2004}. 

The purpose of this brief note is to formulate an abstract $L^2$ Fourier restriction theorem over a certain class of locally compact abelian (LCA) groups. This result provides a unified approach to many of the generalisations of the Stein-Tomas theorem mentioned above (although it should be noted that it certainly fails to recover the deeper and more intricate results in the field such as those of \cite{Ikromov2016} or \cite{Lewko2015}). Moreover, this abstract formulation allows one to easily develop $L^2$ Fourier restriction theory in new settings such as modules over local fields, their associated quotient rings or rings of integers modulo $N$.

Let $G$ be a LCA group with Haar measure $m$ and suppose $G$ is equipped with a one parameter family of translation-invariant balls $\{B^G_{\rho}(x) : x \in G; \rho > 0\}$. The term `balls' is used loosely here: the $B^G_{\rho}(x)$ are simply open sets and need not arise from a metric. They are required, however, to satisfy the following axioms:
\begin{enumerate}[i)]
\item Nesting: $B^G_{\rho}(0) \subseteq B^G_{\rho'}(0)$ for all $0 < \rho \leq \rho'$;
\item Symmetry: $B^G_{\rho}(0) = -B^G_{\rho}(0)$ for all $0 < \rho$;
\item Covering: $\bigcup_{\rho >0} B^{G}_{\rho}(0) = G$;
\item Translation invariance: $B^G_{\rho}(x) = x + B_{\rho}^G(0)$ for all $x \in G$ and $0 < \rho $.
\end{enumerate}
In addition, it is assumed that the balls satisfy the regularity condition
\begin{flalign*}
&\textrm{(R)}\quad m(B_{\rho}^G(0)) \leq C_1\rho^n \textrm{ for all $0<\rho$.}&
\end{flalign*}
Let $\widehat{G}$ denote the Pontryagin dual group and $\hat{m}$ its Haar measure, which is normalised so that the inversion formula (and hence Plancherel's theorem) hold. Suppose $\widehat{G}$ is also equipped with a family of translation-invariant balls $\{B^{\widehat{G}}_{\rho}(\xi) : \xi \in \widehat{G}; \rho > 0\}$ (that is, a system of open sets satisfying i) - iv) above) and, further, that there is a system of real-valued Borel functions $\{\varphi_{\rho}\}_{\rho >0}$ on $G$ such that $\varphi_{\rho} = 1$ on $B^G_{\rho}(0)$, $\mathrm{supp}(\varphi_{\rho}) \subset B^G_{2\rho}(0)$, $\|\varphi_{\rho}\|_{L^{\infty}(G, m)} \leq 1$ and
\begin{flalign*}
&\textrm{(F)}\quad |\hat{\varphi}_{\rho}(\xi)| \leq C_2s^{-n} \textrm{ whenever $-\xi \notin B^{\widehat{G}}_s(0)$ and $s \geq 1/\rho$.}&
\end{flalign*}
Here $\hat{\varphi}_{\rho}$ denotes the Fourier transform of $\varphi_{\rho}$, given by
\begin{equation*}
\hat{\varphi}_{\rho}(\xi) := \int_{G} \varphi_{\rho}(x) \xi(-x)\,\ud m(x) \qquad \textrm{for all characters $\xi \in \widehat{G}$.}
\end{equation*}
In addition to this pointwise estimate for the $\hat{\varphi}_{\rho}$, it is also convenient to assume uniform $L^1(\widehat{G})$-boundedness; explicitly,
\begin{flalign*}
&\textrm{(F$'$)}\quad \int_{\widehat{G}}|\hat{\varphi}_{\rho}(\xi)|\,\ud \hat{m}(\xi) \leq C_3.&
\end{flalign*}
The $\{\varphi_{\rho}\}_{\rho >0}$ can be used to construct a system of operators which can be thought of as `smooth Littlewood-Paley projections'. As such, when all of the above criteria are satisfied, the ensemble $(G, \{B^G_{\rho}\}, \{B^{\widehat{G}}_{\rho}\}, \{\varphi_{\rho}\})$ is referred to as a \emph{Littlewood-Paley system}.

\begin{example}\label{Euclidean example}
The prototypical example is, of course, given by $G = \R^n$ (so that $\widehat{G} \cong \R^n$) and taking the system of balls and dual balls to be simply those induced by the Euclidean metric. Here the Haar measure is Lebesgue measure and condition (R) is immediately satisfied with $C_1 \lesssim_n 1$. A system of projections is given by taking a radially decreasing Schwartz function $\varphi$ satisfying $\varphi(x) = 1$ for $x \in B$ and $\mathrm{supp}(\varphi) \subset 2B$, where $B\subseteq \R^n$ is the unit ball, and defining $\varphi_{\rho}(x) := \varphi(\rho^{-1}x)$ for all $x \in \R^n$; the conditions (F) and (F$'$) are then readily verified with $C_2, C_3 \lesssim_n 1$. The $\varphi_{\rho}$ define a system of Littlewood-Paley projections in the classical sense and applying the forthcoming analysis to this example recovers the results (and methods) of \cite{Bak2011}. Note that one cannot take $\varphi_r$ to be the sharp cutoff function $\chi_{B_{\rho}^G(0)}$: indeed, the lack of regularity of $\chi_{B_{\rho}^G(0)}$ leads to poor Fourier decay estimates. For comparison, in discrete and non-archimedean settings, as considered below, characteristic functions of balls are smooth (in the sense that they admit favourable Fourier decay-type estimates) and in these cases one may take $\varphi_{\rho} := \chi_{B_{\rho}^G(0)}$.
\end{example}

\begin{example}\label{finite field example} If $G = \mathbb{F}_q^n$ is a vector space over a finite field (so that $\widehat{G} \cong \mathbb{F}_q^n$), then one may define
\begin{equation*}
B_{\rho}^G(x) := \left\{\begin{array}{ll}
\emptyset & \textrm{if $0 < \rho < 1$}\\
\{x\} & \textrm{if $1 \leq \rho < q$} \\
\mathbb{F}_q^n & \textrm{if $q \leq \rho < \infty$} 
\end{array}\right. ; \quad B_{\rho}^{\widehat{G}}(\xi) := \left\{\begin{array}{ll}
\emptyset & \textrm{if $0 < \rho < 1/q$}\\
\{\xi\} & \textrm{if $1/q \leq \rho < 1$} \\
\mathbb{F}_q^n & \textrm{if $1 \leq \rho < \infty$} 
\end{array}\right. .
\end{equation*}
Here the Haar measure is counting measure on $G$ and condition (R) is immediately satisfied with $C_1 = 1$. A system of projections is given by $\varphi_{\rho} := \chi_{B_{\rho}^G(0)}$ for all $\rho >0$. The conditions (F) and (F$'$) can be easily verified with $C_2 = C_3 = 1$, noting that here the Haar measure on $\widehat{G}$ is normalised to have mass 1. 
\end{example}

Further examples are discussed in $\S$\ref{examples section}. From Example \ref{finite field example} above one observes that, in general, it is important that the families of balls $\{B^G_{\rho}\}, \{B^{\widehat{G}}_{\rho}\}$ do not necessarily arise from a metric, or even a pseudo-metric. This will also be the case in the basic application where the underlying LCA group $G$ arises from a ring of integers modulo $N$. 

The main result of the article is an abstract $L^2$ restriction theorem for Littlewood-Paley systems. In particular, restriction with respect to some finite (positive) measure $\mu$ on $\widehat{G}$ is investigated. Analogously to the results in the Euclidean setting \cite{Mockenhaupt2000, Mitsis2002, Bak2011}, one assumes that the measure $\mu$ satisfies both a dimensional (or regularity) and Fourier-dimensional hypothesis; in particular, for some $0 < b \leq a < n$ assume the following hold:
\begin{flalign*}
&\textrm{(R$\mu$)}\quad \mu(B_{\rho}^{\widehat{G}}(\xi)) \leq A\rho^a \textrm{ for all $\xi \in \widehat{G}$, and}&
\end{flalign*}
\begin{flalign*}
&\textrm{(F$\mu$)}\quad |\check{\mu}(x)| \leq B\rho^{-b/2} \textrm{ for all $x \notin B^G_{\rho}(0)$.}&
\end{flalign*}
With the various definitions now in place, the main theorem is as follows.

\begin{theorem}\label{abstract theorem} Let $(G, \{B^G_{\rho}\}, \{B^{\widehat{G}}_{\rho}\}, \{\varphi_{\rho}\})$ be a Littlewood-Paley system, $0 < b \leq a < n$ and suppose $\mu$ is a finite measure on $\widehat{G}$ satisfying $(\mathrm{R}\mu)$ and $(\mathrm{F}\mu)$. Then
\begin{equation}\label{abstract Stein Tomas}
\|\hat{f}\|_{L^2(\mu)} \leq C_r\|f\|_{L^r(G)}
\end{equation}
holds for all $1 \leq r \leq r_0$ where 
\begin{equation}\label{exponent formula}
r_0 := \frac{4(n - a) + 2b}{4(n-a) + b}.
\end{equation}
Furthermore, the constant $C_r$ in \eqref{abstract Stein Tomas} depends only on $r, n, C_1, C_2, C_3, A, B, a$ and $b$. 
\end{theorem}

\begin{remark} \begin{enumerate}[1)] \item The proof will in fact show that the Fourier transform satisfies a stronger $L^{r_0,2}(G)-L^{2}(\mu)$ inequality.
\item One may extract an explicit value for the constant appearing in the statement of Theorem \ref{abstract theorem} from the proof presented below. In particular, for $r = r_0$ one may take $C = \bar{C}^{1/2}$ where $\bar{C}$ is a constant of the form
\begin{equation}\label{complicated constant}
\bar{C} = C_{n,a,b}(C_1 + C_2)^{1-\theta} C_3^{(1-\theta)/(2-\theta)} A^{1-\theta} B^{\theta}
\end{equation}
for the exponent $\theta$ given by
\begin{equation}\label{theta}
\theta := \frac{2(n-a)}{2(n-a)+b}.
\end{equation}
In view of applications it is useful to track (at least roughly) the dependence on the constants. This is particularly relevant when considering Fourier restriction in discrete settings such as finite fields, or rings of integers $\Z/N\Z$. In these cases one wishes to prove estimates that are `essentially' independent of the cardinality of the underlying field or ring: see $\S$\ref{examples section}. 
\end{enumerate}
\end{remark}

The theorem is proved by adapting the arguments of \cite{Bak2011} so as to be applicable in the abstract setting of Littlewood-Paley systems. 

This article is structured as follows: in $\S$\ref{examples section} some examples of groups and measures satisfying the hypotheses of the theorem are discussed; the proof of Theorem \ref{abstract theorem} is then given in $\S$\ref{Fourier proof section}. The note concludes with a discussion of some related estimates for the convolution operators $f\mapsto f\ast \mu$ in $\S$\ref{convolution section}.

\begin{acknowledgement} This material is based upon work supported by the National Science Foundation under Grand No. DMS-1440140 whilst the first author was in residence at the Mathematical Sciences Research Institute in Berkeley, California, during the Spring 2017 semester.  
\end{acknowledgement}




\section{Examples}\label{examples section}

In this section examples of Littlewood-Paley systems are discussed, together with some prototypical measures $\mu$ that satisfy (R$\mu$) and (F$\mu$) with favourable values of $a, b, A$ and $B$. For simplicity, the discussion is restricted to measures supported on smooth surfaces or algebraic varieties.

\subsection*{Euclidean spaces} Theorem \ref{abstract theorem} generalises the existing abstract restriction theory of Mockenhaupt \cite{Mockenhaupt2000}, Mitsis \cite{Mitsis2002} and Bak and Seeger \cite{Bak2011}. 


\subsection*{Vector spaces over finite fields} Theorem \ref{abstract theorem} also generalises the basic finite field version of the Stein-Tomas theorem due to Mockenhaupt and Tao \cite{Mockenhaupt2004}. Here it is important to observe that the constants $C_1, C_2, C_3$ can all be chosen independently of the cardinality of the underlying finite field. 
%

\subsection*{Vector spaces over $\Q_p$} Let $p$ be a fixed odd prime and consider the field of $p$-adic numbers $\Q_p$ with $p$-adic absolute value $|\,\cdot\,|_p$. The vector space $G:=\Q_p^n$ is self-dual as a LCA group and both $G$ and $\widehat{G}$ are endowed with the family of (clopen) balls $B_{\rho}(x) := \{x \in \Q_p^n : \|x\| \leq \rho\}$ induced by the $\ell^{\infty}$-norm
\begin{equation*}
\|x\| := \max_{1 \leq j \leq n}|x_j|_p \quad \textrm{for all $x = (x_1, \dots, x_n) \in \Q_p^n$.}
\end{equation*}
The Haar measures on both $\Q_p^n$ and the dual group are normalised so that the unit ball $\Z_p^n := \{x \in \Q_p^n : \|x\| \leq 1\}$ has measure 1. The regularity property (R) then holds with $C_1 = 1$.

Fix an additive character $e \colon \Q_p \to \mathbb{T}$ such that $e$ restricts to the constant function 1 on $\Z_p$ and to a non-principal character on $p^{-1}\Z_p$. Then for any integrable $f \colon \Q_p^n \to \C$ the Fourier transform $\hat{f}$ is given by
\begin{equation*}
\hat{f}(\xi) := \int_{\Q_p^n} f(x)e(- x \cdot \xi ) \,\ud m(x) \qquad \textrm{for all $\xi \in \Q_p^n$.}
\end{equation*}
In particular, defining $\varphi_{\rho} := \chi_{B_{\rho}(0)}$ for $\rho > 0$ one may easily verify that $\hat{\varphi}_{\rho}(\xi) = p^{-n\nu} \chi_{B_{p^{\nu}}(0)}(\xi)$ where $\nu \in \Z$ is the smallest integer such that $p^{-\nu} \leq \rho$. The conditions (F) and (F$'$) immediately follow with $C_2 = C_3 = 1$. It is remarked that, by the non-archimedean nature of the absolute value, the $\varphi_{\rho}$ are smooth functions on $\Q_p^n$ and, moreover, are natural $p$-adic analogues of Schwartz functions (in particular, they belong to the \emph{Schwartz-Bruhat} class of functions on $\Q_p^n$: see  \cite{Bruhat1961, Osborne1975, Taibleson1975}).

Let $h  \colon \Z_p^{n-1} \to \Z_p$ be the mapping given by $h(\omega) := \omega_1^2 + \dots + \omega_{n-1}^2$ and consider the paraboloid 
\begin{equation*}
\Sigma := \{(\omega, h(\omega)) : \omega \in \Z_p^{n-1}\} \subseteq \Q_p^n.
\end{equation*}
Take $\mu$ to be the measure on $\Sigma$ given by the push-forward of the Haar measure on $\Z_p^{n-1}$ under the graphing function $\omega \mapsto (\omega, h(\omega))$. The condition (R$\mu$) is readily verified for $\mu$ with $A = 1$ and $a=n-1$. On the other hand, the inverse Fourier transform of the measure is given by $\check{\mu}(x) = \prod_{j=1}^{n-1} G(x_j, x_n)$ where
\begin{equation*}
G(a,b) := \int_{\Z_p} e(at + bt^2)\,\ud t \qquad \textrm{for $a, b \in \Q_p$.}
\end{equation*}
The integral $G(a,b)$ can written in terms of classical Gauss sums and thereby evaluated or, alternatively, one may analyse $G(a,b)$ directly using the basic algebraic properties of the character $e$. In either case, it is not difficult to deduce that 
\begin{equation*}
|G(a,b)|  = \left\{\begin{array}{ll}
|b|_p^{-1/2} & \textrm{if $|a|_p \leq |b|_p$}\\
0 & \textrm{otherwise}
\end{array} \right. \qquad \textrm{for all $a, b \in \Q_p$ with $\max\{|a|_p,|b|_p\} > 1$.} 
\end{equation*}
From these observations it follows that 
\begin{equation*}
|\check{\mu}(x)| \leq \|x\|^{-(n-1)/2} \qquad \textrm{for all $x \in \Q_p^n$}
\end{equation*}
and therefore (F$\mu$) holds with $B=1$ and $b = n-1$. 

Combining these observations with Theorem \ref{abstract theorem} produces a $p$-adic variant of the classical Stein-Tomas theorem, which shares the $r_0 = 2(n+1)/(n+3)$ numerology of the Euclidean case.

\subsection*{Modules over rings of integers $\Z/p^{\alpha}\Z$} Let $N \in \N$ and consider the module $G := [\Z/N\Z]^n$. For $k \in \N$ define a function $\|\,\cdot\,\| \colon [\Z/N\Z]^k \to \N$ by setting 
\begin{equation*}
\|\vec{x}\| := \frac{N}{\gcd(x_1, \dots, x_k, N)} \qquad \textrm{for all $\vec{x} = (x_1, \dots, x_k) \in [\Z/N\Z]^k$.} 
\end{equation*}
The image of this `norm' is thought of as a set of available scales in the module. It is natural to compare the scales under the division ordering $\preceq$, defined by $a \preceq b$ for $a,b\in \N$ if and only if $a \mid b$. One may isolate the elements lying at a given scale by defining
\begin{equation*}
\mathcal{B}_d := \big\{ \vec{x} \in [\Z/N\Z]^n : \|\vec{x}\| \preceq d \big\} \quad \textrm{for all divisors $d \mid N$.}
\end{equation*}

Fix an odd prime $p$ and now specialise to the case $N = p^{\alpha}$ for some $\alpha \in \N$. In this situation the set of available scales is given by $\{1, p, \dots, p^{\alpha}\}$, which is totally ordered under $\preceq$. Define a collection of balls on $G:=[\Z/p^{\alpha}\Z]^n$ by $B_{\rho}^G(\vec{x}\,) := \vec{x} + \mathcal{B}_{p^{\nu}}$ where $0 \leq \nu \leq \alpha$ is the largest value for which $p^{\nu} \leq \rho$ (if $0 < \rho < 1$, then $B_{\rho}^G(\vec{x}\,) := \emptyset$). These balls do not arise from a metric, but nevertheless the satisfy the crucial properties i) - iv) listed in the introduction. Furthermore, the Haar measure on $[\Z/p^{\alpha}\Z]^n$ is simply counting measure and the regularity property (R) therefore holds with $C_1 = 1$. A system of dual balls on $\widehat{G}$ is given by $B_{\rho}^{\widehat{G}}(\vec{\xi}\,) := \vec{\xi} + \mathcal{B}_{p^{\alpha-\nu}}$ where now $0 \leq \nu \leq \alpha$ is the smallest value for which $p^{\nu} \geq 1/\rho$ (if $0 < \rho < p^{-\alpha}$, then $B_{\rho}^{\widehat{G}}(\vec{\xi}\,) := \emptyset$). Taking $\varphi_{\rho} := \chi_{B_{\rho}^G(\vec{0}\,)}$ the properties (F) and (F$'$) are both easily seen to hold with $C_2 = C_3 = 1$. 
 
Let $h \colon [\Z/p^{\alpha}\Z]^{n-1}  \to \Z/p^{\alpha}\Z$ be the mapping given by $h(\vec{\omega}) := \omega_1^2 + \dots + \omega_{n-1}^2$ and consider the paraboloid
\begin{equation*}
\Sigma := \{(\vec{\omega}, h(\vec{\omega})) : \vec{\omega} \in [\Z/p^{\alpha}\Z]^{n-1}\} \subseteq [\Z/p^{\alpha}\Z]^{n}.
\end{equation*}
If $\mu$ denotes the normalised counting measure on $\Sigma$, then it is immediate that $\mu$ satisfies (R$\mu$) with $A=1$ and $a = n-1$. The Fourier transform
\begin{equation*}
\check{\mu}(\vec{x}\,) = \frac{1}{p^{(n-1)\alpha}} \sum_{\omega \in [\Z/p^{\alpha}\Z]^{n-1}} e^{2\pi i ( x'\cdot \vec{\omega} + x_nh(\vec{\omega}))/p^{\alpha}} \quad \textrm{for $\vec{x} = (x',x_n) \in [\Z/p^{\alpha}\Z]^{n}$}
\end{equation*}
can be evaluated via the classical formulae for Gauss sums. In particular, it is not difficult to show that
\begin{equation*}
|\check{\mu}(\vec{x}\,)| \leq \|\vec{x}\,\|^{-(n-1)/2},
\end{equation*}
which implies that (F$\mu$) holds with $B=1$ and $b=n-1$. See \cite{Hickman} for further details.

Applying Theorem \ref{abstract theorem}, one deduces that the inequality
\begin{equation*}
\Big(\frac{1}{\#\Sigma}\sum_{\vec{\xi} \in \Sigma} |\hat{F}(\vec{\xi}\,)|^2 \Big)^{1/2} \lesssim_n \|F\|_{\ell^r([\Z/p^{\alpha}\Z]^n)}
\end{equation*}
holds for all $1 \leq r \leq 2(n+1)/(n+3)$. The important observation here is that the implied constant in this estimate is independent of both $p$ and $\alpha$ (and therefore the cardinality of the underlying ring).\footnote{Indeed, the inequality 
\begin{equation*}
\Big(\frac{1}{\#\Sigma}\sum_{\xi \in \Sigma} |\hat{F}(\xi)|^s \Big)^{1/s} \lesssim_{n,p,\alpha} \|F\|_{\ell^r([\Z/p^{\alpha}\Z]^n)}
\end{equation*}
trivially holds for all Lebesgue exponents $1 \leq r, s \leq \infty$ (with a constant which now depends on $p$ and $\alpha$) as a consequence of the Riemann-Lebesgue lemma and the equivalence of norms on finite-dimensional vector spaces.} Thus, this result is a $[\Z/p^{\alpha}\Z]^n$-analogue of certain finite field restriction estimates of Mockenhaupt and Tao \cite{Mockenhaupt2004}. 

The analysis of this discrete example has many similarities with the continuous $p$-adic example. In fact, restriction theory over $\Q_p$ is equivalent to restriction theory over $[\Z/p^{\alpha}\Z]^n$ in a precise sense. In particular, there is a `correspondence principle', which is a manifestation of the uncertainty principle, that allows one to `lift' restriction problems over the discrete rings $\Z/p^{\alpha}\Z$ to the continuous setting of $\Q_p$. This lifting procedure is discussed in detail in \cite{Hickman2015} and \cite{Hickman}.

\subsection*{Vector spaces over local fields and modules over their quotient rings} The two previous examples can be generalised to the setting of non-archimedean local fields. Let $K$ be a field with a discrete non-archimedean absolute value $|\,\cdot\,|_K$, suppose $\pi \in K$ is a choice of uniformiser and let $\mathfrak{o}:= \{x \in K : |x|_K \leq 1\}$ denote the ring of integers of $K$. Assume that the residue class field $\mathfrak{o}/\pi \mathfrak{o}$ is finite. For the details of the relevant definitions see, for instance, \cite{Lang1984}, \cite{Lang1994} or \cite{Taibleson1975}. Generalising the $p$-adic example, any finite-dimensional vector space $K^n$ can be endowed with a natural Littlewood-Paley system by taking the balls to be those induced by the $\ell^{\infty}$-norm on $K^n$ and the projections $\varphi_{\rho}$ to be characteristic functions of balls. It is remarked that, by the non-archimedean nature of the absolute value, these $\varphi_{\rho}$ are in fact smooth functions. Similarly, generalising the $\Z/p^{\alpha}\Z$ example, for each $\alpha \in \N$ the module $[\mathfrak{o}/\pi^{\alpha}\mathfrak{o}]^n$ can also be endowed with a natural Littlewood-Paley system. The restriction theories over the vector space $K^n$ and over the modules $[\mathfrak{o}/\pi^{\alpha}\mathfrak{o}]^n$ are in some sense equivalent via a correspondence principle which extends that described above. The details may be found in \cite{Hickman2015}. 

It is well-known that any field $K$ satisfying the above properties is isomorphic to either a finite extension of $\Q_p$ for some prime $p$ or the field $\mathbb{F}_q((X))$ of formal Laurent series over a finite field $\mathbb{F}_q$. The local fields $\mathbb{F}_q((X))$ are particularly well-behaved spaces which act as simplified models of Euclidean space. For instance, Fourier analysis over $\mathbb{F}_2((X))$ corresponds to the study of Fourier-Walsh series, which has played a prominent r\^ole as a model for problems related to Carleson's theorem and time-frequency analysis \cite{DiPlinio2014, Do2012}. Recently there has been increased interest in local field variants of other problems in Euclidean harmonic analysis and geometric measure theory, focusing on the Kakeya conjecture \cite{Ellenberg2010, Caruso, Dummit2013, Fraser2016, Hickman}. This has stemmed from Dvir's solution \cite{Dvir2009} to Wolff's finite field Kakeya conjecture, which has led to progress on the original Euclidean problem \cite{Guth2010, Carbery2013, Guth2016}. It is natural to also consider local field analogues of the restriction problem; this topic is investigated further in \cite{Hickman2015, Hickman}.

\subsection*{Modules over rings of integers $\Z/N\Z$} Let $N \in \N$ and consider the ring of integers $\Z/N\Z$. If $N$ is not a power of a fixed prime, but has multiple distinct prime factors, then the set of available scales for $\Z/N\Z$ is only partially ordered under $\preceq$. This introduces additional difficulties when one attempts to generalise the constructions described in the $\Z/p^{\alpha}\Z$ case. In particular, in order to ensure the nesting property, the balls $B_{\rho}^G(\vec{x}\,)$ and $B_{\rho}^{\widehat{G}}(\vec{\xi}\,)$ in $[\Z/N\Z]^n$ are now defined by
\begin{equation*}
B_{\rho}^G(\vec{x}\,) : = \vec{x} + \bigcup_{d\mid N : d \leq \rho} \mathcal{B}_d \quad \textrm{and} \quad B_{\rho}^{\widehat{G}}(\vec{\xi}\,) : = \vec{\xi} + \bigcup_{d\mid N : d \geq 1/\rho} \mathcal{B}_{N/d}.
\end{equation*} 
The verification of the properties (R), (F) and (F$'$) for these balls is more involved and an $\varepsilon$-loss in $N$ must, in general, be included in the constants. The details are discussed in \cite{Hickman}, where a theory of Fourier restriction over such rings of integers is systematically developed. The partially ordered scale structure on $\Z/N\Z$ tends to make the analysis more involved in this setting than over $\R^n$ (where the scales are, of course, totally ordered), and typically the arguments require additional number-theoretic input \cite{Hickman2, Hickman}.




\section{Proof of Theorem \ref{abstract theorem}}\label{Fourier proof section}

The main ingredient is a convolution inequality.

\begin{proposition}\label{convolution} For $G$ and $\mu$ as in the statement of Theorem \ref{abstract theorem} define $Tf := f \ast \check{\mu}$. If
\begin{equation*}
\sigma := \frac{2(n-a+b)(2(n-a)+b)}{2(n-a+b)(2(n-a)+b) -b^2}, \qquad \tau := \frac{2(n-a+b)}{b},
\end{equation*}
then whenever $\sigma < p < \tau'$ and $q$ satisfies $1/p - 1/q = 2(n-a)/(2(n-a) + b)$, the estimate 
\begin{equation*}
\|Tf\|_{L^{q,s}(G)} \lesssim_{p,s} \bar{C}\|f\|_{L^{p,s}(G)}
\end{equation*}
holds for any $0 < s \leq \infty$. Here $\bar{C}$ is the expression appearing in \eqref{complicated constant}.
\end{proposition}

Theorem \ref{abstract theorem} is an immediate consequence of this estimate. 

\begin{proof}[Proof (of Theorem \ref{abstract theorem})] Note that $(p, q) := (r_0, r_0')$ satisfies the hypotheses of Proposition \ref{convolution}. By the Lorentz space version of H\"older's inequality together with a duality argument,
\begin{align*}
\int_{\widehat{G}}|\hat{f}(\xi)|^2\,\ud \mu(\xi) &= \int_{G} f(x) \overline{Tf(x)} \,\ud m(x) \\
&\leq \|f\|_{L^{r_0,2}(G)} \|Tf\|_{L^{r_0',2}(G)} \lesssim_{p,s} \bar{C}\|f\|_{L^{r_0,2}(G)}^2.
\end{align*}
Interpolating against the trivial $L^1(G)-L^{\infty}(\mu)$ inequality concludes the proof.
\end{proof}

Turning to the proof of Proposition \ref{convolution}, the first step is, in fact, to prove the restricted weak-type version of the estimate \eqref{abstract Stein Tomas} for $r = r_0$. This is achieved via (an abstraction of) an $L^2$ restriction argument due to A. Carbery. The weak version of the Stein-Tomas theorem can then be applied to bound the convolution operator.

\begin{lemma}\label{restricted weak type lemma} Under the hypotheses of Theorem \ref{abstract theorem},\footnote{In fact, the hypotheses can be slightly weakened: here the symmetry property ii) of the balls and $L^1$ estimate (F$'$) are not required.} the restricted weak-type estimate
\begin{equation*}
\|\hat{\chi}_E\|_{L^2(\mu)} \lesssim_{n,a}(C_1 + C_2)^{(1-\theta)/2} A^{(1-\theta)/2}B^{\theta/2} \|\chi_E\|_{L^{r_0}(G)}
\end{equation*}
holds for all Borel sets $E \subset G$.
\end{lemma}

\begin{proof} Decompose the measure $\mu$ by writing $\mu = \mu_1 + \mu_2$ where 
\begin{equation}\label{mu decomposition}
\check{\mu}_1 :=  \varphi_{\rho}\cdot\check{\mu} \quad \textrm{and} \quad \check{\mu}_2 := (1-\varphi_{\rho}) \cdot \check{\mu}
\end{equation}
for fixed value of $\rho > 0$ chosen so as to satisfy the later requirements of the proof. Thus, $T = T_1 + T_2$ where $T_jf := f \ast \check{\mu}_j$ for $j = 1,2$. 

Fixing a Borel set $E \subseteq \widehat{G}$ observe, by duality and H\"older's inequality, that
\begin{equation}\label{restricted weak type 1}
\int_{\widehat{G}} |\hat{\chi}_E(\xi)|^2\,\ud \mu(\xi) \leq  \|T_1\chi_E\|_{L^{2}(G)}m(E)^{1/2} + \|T_2\chi_E\|_{L^{\infty}(G)}m(E).
\end{equation}
Since $\mu_1 = \hat{\varphi}_{\rho} \ast \mu$, it follows that
\begin{align*}
\mu_1(\xi) &= \int_{B^{\widehat{G}}_{1/\rho}(\xi)} \hat{\varphi}_{\rho}(\xi - \eta)\,\ud \mu(\eta) + \sum_{k=1}^{\infty} \int_{B^{\widehat{G}}_{2^k/\rho}(\xi) \setminus B^{\widehat{G}}_{2^{k-1}/\rho}(\xi)} \hat{\varphi}_{\rho}(\xi - \eta)\,\ud \mu(\eta)\\
&=: \RN{1} + \RN{2}
\end{align*}
Applying the Riemann-Lebesgue estimate $\|\hat{\varphi}_{\rho}\|_{L^{\infty}(\widehat{G})} \leq m(B^G_{2\rho}(0))$ together with the hypotheses (R) and (R$\mu$), one deduces that
\begin{equation*}
|\RN{1}| \leq m\big(B^G_{2\rho}(0)\big) \mu\big(B^{\widehat{G}}_{1/\rho}(\xi)\big) \leq 2^nC_1A\rho^{n-a}.
\end{equation*}
Furthermore, for any $k \in \N$ the condition (F) implies that 
\begin{equation*}
|\hat{\varphi}_{\rho}(\xi - \eta)| \leq C_2 2^{-(k-1)n}\rho^{n} \qquad \textrm{for all $\eta \notin B^{\widehat{G}}_{2^{k-1}/\rho}(\xi)$}
\end{equation*}
and so
\begin{equation*}
|\RN{2}| \leq C_2\bigg(\sum_{k=1}^{\infty} 2^{-(k-1)n}\mu\big(B^{\widehat{G}}_{2^k/\rho}(\xi)\big)\bigg) \rho^{n} 
\leq 2^nC_2\bigg(\sum_{k=1}^{\infty} 2^{-(n-a)k}\bigg)A \rho^{n-a}.
\end{equation*}
Combining these observations,
\begin{equation}\label{mu1 estimate}
\|\mu_1\|_{L^{\infty}(\widehat{G})} \leq 2^n\big(C_1+(2^{n-a}-1)^{-1}C_2\big) A\rho^{n-a} \lesssim_{n,a} (C_1 + C_2)A\rho^{n-a}
\end{equation}
and so
\begin{equation}\label{mu1 estimate 2}
\|T_1\chi_E\|_{L^2(G)} = \|\mu_1\hat{\chi}_E\|_{L^2(\widehat{G})} \lesssim_{n,a} (C_1 + C_2)Am(E)^{1/2}\rho^{n-a}.
\end{equation}

On the other hand, since $\mathrm{supp}(1 - \varphi_{\rho}) \subseteq G \setminus B^G_{\rho}(0)$, it follows from (F$\mu$) that 
\begin{equation}\label{mu2 estimate}
\|\check{\mu}_2\|_{L^{\infty}(G)} \leq 2B\rho^{-b/2}
\end{equation}
 and hence
\begin{equation}\label{restricted weak type 3}
\|T_2\chi_E\|_{L^{\infty}(G)} \leq \|\check{\mu}_2\|_{L^\infty(G)}\|\chi_E\|_{L^1(\widehat{G})} \lesssim Bm(E)\rho^{-b/2}.
\end{equation}

Combining \eqref{restricted weak type 1}, \eqref{mu1 estimate 2} and \eqref{restricted weak type 3} one concludes that
\begin{equation*}
\|\hat{\chi}_E\|_{L^2(\mu)}^2 \lesssim_{n,a} (C_1 + C_2)Am(E)\rho^{n-a} + Bm(E)^2\rho^{-b/2}
\end{equation*}
Thus, choosing $\rho$ so that $\rho^{n-a+b/2} \sim_{n,a} (C_1+C_2)^{-1}A^{-1}B m(E)$ and recalling the definition \eqref{theta}, the desired inequality follows.
\end{proof}

\begin{proof}[Proof (of Proposition \ref{convolution})] Since $T$ is essentially self-adjoint\footnote{In particular, $T^*g = g \ast \check{\tilde{\mu}}$ where $\tilde{\mu}$ is the measure on $\widehat{G}$ given by $\tilde{\mu}(E) := \mu(-E)$ for all Borel sets $E \subseteq \widehat{G}$. Note (R$\tilde{\mu}$) and (F$\tilde{\mu}$) hold if and only if (R$\mu$) and (F$\mu$) hold with the same constants $A$ and $B$ and so the subsequent arguments apply equally to $T$ and $T^*$.}, it suffices to show that $T$ is of restricted weak-type $(\sigma, \tau)$. Indeed, it then follows that $T$ also of restricted weak-type $(\tau', \sigma')$ and the desired result is then deduced by interpolating between these estimates. 

Decompose $T = T_1 + T_2$ as above; although the same notation is used, it is understood that this decomposition is made with respect to a new value of $\rho$, chosen so as to satisfy the later requirements of the proof. Applying \eqref{mu1 estimate}, one observes that
\begin{align}
\nonumber
\|T_1\chi_E\|^2_{L^2(G)} &= \int_{\widehat{G}} |\hat{\chi}_E(\xi)|^2|\mu_1(\xi)|^2\,\ud \hat{m}(\xi) \\
\nonumber
&\lesssim_{n,a} (C_1 + C_2)A\rho^{n-a} \int_{\widehat{G}} |\hat{\chi}_E(\xi)|^2|\mu_1(\xi)|\,\ud \hat{m}(\xi) \\
\nonumber
&\leq (C_1 + C_2)A\rho^{n-a} \int_{\widehat{G}}\int_{\widehat{G}} |\hat{\chi}_E(\xi+\eta)|^2\ud\mu(\eta)|\hat{\varphi}_{\rho}(\xi)|\,\ud \hat{m}(\xi) \\
\label{improved T_1}
&\lesssim_{n,a,b} (C_1 + C_2)^{2-\theta}C_3A^{2-\theta}B^{\theta} m(E)^{2/r_0}\rho^{n-a},
\end{align}
where the final inequality is due to Lemma \ref{restricted weak type lemma} and (F$'$). If $F \subseteq G$ is any Borel set, then
\begin{align*}
\langle T\chi_E\,,\, \chi_F \rangle &= \langle T_1\chi_E\,,\, \chi_F \rangle +\langle T_2\chi_E\,,\, \chi_F \rangle \\
&\leq \|T_1\chi_E\|_{L^2(G)}m(F)^{1/2} + \|T_2\chi_E\|_{L^{\infty}(G)}m(F).
\end{align*}
Thus, as a consequence of \eqref{restricted weak type 3} and \eqref{improved T_1}, the right-hand side of the above expression is dominated by
\begin{equation*}
(C_1 + C_2)^{1-\theta/2}C_3^{1/2}A^{1-\theta/2}B^{\theta/2} m(E)^{1/r_0}m(F)^{1/2}\rho^{(n-a)/2} +Bm(E)m(F)\rho^{-b/2}.
\end{equation*}
Choosing 
\begin{equation*}
\rho^{(n-a +b)/2} \sim_{n,a,b} ((C_1 + C_2)AB^{-1})^{-(1-\theta/2)}C_3^{-1/2}m(E)^{1/r_0'}m(F)^{1/2}
\end{equation*}
yields the estimate
\begin{equation*}
\langle T\chi_E\,,\, \chi_F \rangle \leq \bar{C} m(E)^{1/\sigma}m(F)^{1/\tau'}
\end{equation*}
where $\sigma, \tau$ and $\bar{C}$ are as in the statement of the proposition. In particular, $T$ is of restricted weak-type $(\sigma, \tau)$, as required.
\end{proof}




\section{Some remarks on convolution operators}\label{convolution section}

Recall that Theorem \ref{abstract theorem} was a direct consequence of an estimate for the convolution operator $f \mapsto \check{\mu} \ast f$. Using the method of proof of Lemma \ref{restricted weak type lemma}, one may also obtain $L^q(\widehat{G})-L^r(\widehat{G})$ estimates for the related convolution operator $f \mapsto \mu \ast f$.

\begin{lemma}\label{abstract convolution lemma}
Let $(G, \{B^G_r\}, \{B^{\widehat{G}}_r\}, \{\varphi_r\})$ be a Littlewood-Paley system, $0 < b \leq a < n$ and suppose $\mu$ is a probability measure on $\widehat{G}$ satisfying $(\mathrm{R}\mu)$ and $(\mathrm{F}\mu)$. If $\mathcal{T}$ denotes the closed triangle with vertices $\{(0,0), (1,1), (1/r_0,1/s_0)\}$ where
\begin{equation*}
r_0 := \frac{2(n-a) + b}{n-a+b} \qquad s_0 := \frac{2(n-a)+b}{n-a},
\end{equation*}
then
\begin{equation}\label{abstract convolution}
\|f \ast \mu\|_{L^s(\widehat{G})} \leq C\|f\|_{L^r(\widehat{G})}
\end{equation}
holds whenever $(1/r,1/s) \in \mathcal{T}\setminus\{(1/r_0, 1/s_0)\}$. Furthermore, the constant $C$ depends only on $n, C_1, C_2, A, B, a$ and $b$.
\end{lemma}

This is a partial extension of a classical generalised Radon transform estimate due to Littman \cite{Littman1973}. The latter treats the case where $\mu$ is a smooth, compactly support density supported on a hypersurface in $\R^n$, under the assumption that the hypersurface has non-vanishing Gaussian curvature on the support of $\mu$. In this case Littman \cite{Littman1973} establishes \eqref{abstract convolution} for the sharp range $(1/r, 1/s) \in \mathcal{T}$, including the $(1/r_0, 1/s_0)$ endpoint. Lemma \ref{abstract convolution lemma} is known to hold (together with the endpoint estimate) in Euclidean space for general measures satisfying (R$\mu$) and (F$\mu$), although as far as the authors are aware this has not appeared in print (see, however, \cite{Gressman}). The finite field case has also been studied \cite{Carbery2008}.

\begin{proof}[Proof (of Lemma \ref{abstract convolution lemma})] 
Since $\mu$ is a probability measure it follows that the convolution operator is bounded on $L^r(\widehat{G})$ for all $1 \leq r \leq \infty$. It therefore suffices to prove that $f \mapsto f \ast \mu$ satisfies a restricted weak-type $(r_0, s_0)$ inequality. 

Decompose $\mu$ by writing $\mu = \mu_1 + \mu_2$ where the $\mu_j$ are as defined in \eqref{mu decomposition}. Once again, $\rho > 0$  is a fixed value chosen so as to satisfy the later requirements of the proof. 

For Borel sets $E, F \subset G$ it follows that
\begin{align*}
\langle \mu \ast \chi_E\,,\, \chi_F \rangle &\leq \|\mu_1 \ast \chi_E\|_{L^{\infty}(\widehat{G})} \hat{m}(F) +  \|\mu_2 \ast \chi_E\|_{L^2(\widehat{G})} \hat{m}(F)^{1/2}  \\
&\leq \|\mu_1\|_{L^{\infty}(\widehat{G})} \hat{m}(E)\hat{m}(F) + \|\check{\mu}_2\|_{L^{\infty}(G)} \hat{m}(E)^{1/2}\hat{m}(F)^{1/2}.
\end{align*}
The proof of Lemma \ref{restricted weak type lemma}, and in particular \eqref{mu1 estimate} and \eqref{mu2 estimate},  now imply that 
\begin{equation*}
\langle \mu \ast \chi_E\,,\, \chi_F \rangle \lesssim_{n,a}(C_1 + C_2)A\rho^{n-a}m(E)m(F) + B\rho^{-b/2}m(E)^{1/2}m(F)^{1/2}.
\end{equation*}
Thus, choosing $\rho$ so that $\rho^{n-a+b/2} \sim_{n,a} (C_1+C_2)^{-1}A^{-1}B m(E)^{-1/2}m(F)^{-1/2}$, one concludes that
\begin{equation*}
\langle \mu \ast \chi_E\,,\, \chi_F \rangle \lesssim_{n,a,b}(C_1+C_2)^{1-\theta}A^{1-\theta}B^{\theta} m(E)^{r_0}m(F)^{s_0'}
\end{equation*}
for $\theta$ as defined in \eqref{theta}, as required. 
\end{proof}

As a final remark, when $G = \R^n$ the strong-type $(r_0,s_0)$ estimate can be obtained for $f \mapsto f \ast \mu$ by augmenting the above argument with standard inequalities for the Littlewood-Paley square function. It would be interesting to understand whether the endpoint estimate holds in general, given that the spaces in question do not fall under any existent Calder\'on-Zygmund theory.




\bibliography{Reference}
\bibliographystyle{amsplain}
\end{document}